\documentclass[11pt,a4paper]{amsart}

\usepackage[latin1]{inputenc}
\usepackage[english]{babel}
\usepackage{amsmath}
\usepackage{bm}
\usepackage{amssymb}
\usepackage{graphicx}
\usepackage{braket}
\usepackage[pdftex,plainpages=false,colorlinks,hyperindex,bookmarksopen,linkcolor=red,citecolor=blue,urlcolor=blue]{hyperref}
\usepackage{cite}
\usepackage{mathrsfs}
\usepackage{epstopdf}
\usepackage{braket}
\usepackage[hmargin=3cm,vmargin={3.5cm,4cm}]{geometry}

\theoremstyle{theorem}
\newtheorem{theorem}{Theorem}[section]
\newtheorem{lemma}{Lemma}[section]
\newtheorem{proposition}{Proposition}[section]

\theoremstyle{definition}                           %stile roman
\theoremstyle{remark}                             %stile per osservazioni
\newtheorem{remark}{Remark}[section]              %definizione ambiente osservazione
\newtheorem{example}{Example}[section]

\usepackage{color}

\usepackage{mathtools,slashed}

\newcommand{\be}{\begin{eqnarray}}
\newcommand{\ee}{\end{eqnarray}}
\DeclareMathOperator{\real}{Re}

\def\Cset{{\mathbb{C}}}

\def\Nset{{\mathbb{N}}}
\def\eu{{\ensuremath{\mathrm{e}}}}

\def\du{\,{\ensuremath{\mathrm{d}}}}
\def\texttiny#1{{\text{\tiny{#1}}}}
\def\DC{{}^{\texttiny{C}}\! D}
\def\DR{{}^{\texttiny{RL}}\! D}
\def\DG{{}^{\texttiny{GL}}\! D}

\def\fR{{f}^{\text{{\scriptsize{R}}}}}
\def\fC{{f}^{\text{{\scriptsize{C}}}}}
\def\DABC{{}^{\texttiny{ABC}}\! D}

\def\DCF{{}^{\texttiny{CF}}\! D}
\def\ICF{{}^{\texttiny{CF}}\! J}
\def\IAB{{}^{\texttiny{AB}}\! J}
\def\IRL{{}^{\texttiny{RL}}\! J}
\def\JPr{{\bm{\mathcal{J}}}}     % Integrale Prabhakar
\def\DPrC{{}^{\texttiny{C}}\!{\bm{\mathcal{D}}}}   % Derivata Caputo-Prabhakar
\renewcommand{\Re}{\operatorname{Re}}

\def\Dphi{D_{\phi}}
\def\Jpsi{J_{\psi}}

\allowdisplaybreaks
\begin{document}

\title[Why fractional derivatives with \dots]{Why fractional derivatives with \\ [3pt] nonsingular kernels should not be used}
	
		\author[K. Diethelm]{Kai Diethelm${}^{1}$}
		\address{${}^{1}$ 
		Fakult\"at Angewandte Natur- und Geisteswissenschaften,
 		University of Applied Sciences W\"urzburg-Schweinfurt,
 		Ignaz-Sch\"on-Str.\ 11, 97421 Schweinfurt, GERMANY
		and GNS mbH Gesellschaft f\"ur numerische Simulation mbH,
		Am Gau{\ss}berg 2, 38114 Braunschweig, GERMANY}	
 		\email{kai.diethelm@fhws.de}
		
		\author[R. Garrappa]{Roberto Garrappa${}^{2}$}
		\address{${}^{2}$ 
		Department of Mathematics,
 		University of Bari, 
		Via E. Orabona 4, 70126 Bari, ITALY
		and 
		the INdAM Research group GNCS}	
 		\email{roberto.garrappa@uniba.it}		

	    \author[A. Giusti]{Andrea Giusti${}^{3}$}
		\address{${}^{3}$ 
		Bishop's University,
		Physics $\&$ Astronomy Department, 
		2600 College Street, Sherbrooke, J1M 1Z7,
		QC	CANADA}	
 		\email{agiusti@ubishops.ca}

 		\author[M. Stynes]{Martin Stynes${}^{4}$}
		\address{${}^{4}$ 
		Applied and Computational Mathematics Division,
		Beijing Computational Science Research Center,
		Beijing 100193, CHINA}	
 		\email{m.stynes@csrc.ac.cn}
	
\begin{abstract}
In recent years, many papers discuss the theory and applications of new fractional-order derivatives that are constructed by replacing the singular kernel of the Caputo or Riemann-Liouville derivative by a non-singular (i.e., bounded) kernel. It will be shown here, through rigorous mathematical reasoning, that these non-singular kernel derivatives suffer from several drawbacks which should forbid their use. They fail to satisfy the fundamental theorem of fractional calculus since they do not admit the existence of a corresponding convolution integral of which the derivative is the left-inverse; and the value of the derivative  at the initial time $t=0$ is always zero, which imposes an unnatural restriction on the differential equations and models where these derivatives can be used. For the particular cases of the so-called Caputo-Fabrizio and Atangana-Baleanu derivatives, it is shown that when this restriction  holds the derivative can be simply expressed in terms of integer derivatives and standard Caputo fractional derivatives, thus demonstrating that these derivatives contain nothing new.
\end{abstract}

\thanks{\textbf{In}: \emph{Fract. Calc. Appl. Anal.}, Vol. \textbf{23}, No 3 (2020), pp. 610--634, 
DOI: 10.1515/fca-2020-0032 at \href{https://www.degruyter.com/view/j/fca}{https://www.degruyter.com/view/j/fca}}

    \maketitle

\section{Introduction}\label{sec:1}

Fractional calculus, namely the study of the generalization of the standard theory of calculus to derivatives and integrals of non-integer orders, has attracted much attention in recent years from different disciplines. It is not only of interest to mathematicians;  its success derives from its proven effectiveness in accurately describing innumerable physical phenomena, ranging from biophysics to astrophysics.

Throughout the history of this theory, several definitions for non-integer order operators have been proposed; each one is an attempt to extend the classical notions of integral and derivative. Among these proposals, two particular ones have stood the test of time and are now universally accepted: the celebrated works of Bernhard Riemann and Joseph Liouville, and its modification suggested by Mkhitar Dzhrbashyan.  Most notably, the latter turned out to be equivalent to the operator independently inferred by Michele Caputo as a direct result of the generalization of the standard Laplace transform of ordinary derivatives to the fractional regime. The proposal of Riemann and Liouville, which started the entire field of fractional calculus, was based on performing an analytic continuation of Cauchy's formula for repeated integration. This operator is now known as the Riemann-Liouville (RL) fractional integral. Similarly, starting from Cauchy's integral formula for the $n$th derivative of an analytic function, one can then provide a definition of fractional derivative. This definition of derivative is however well posed only when applied to some very well-behaved functions, so it is naturally desirable to extend the class of permissible functions.

This extension can be achieved by defining a fractional derivative via an integro-differential operator with a locally absolutely integrable kernel, in the form of the well-known Riemann-Liouville (RL) fractional derivative;  see~\eqref{eq:RL_Derivative} in Appendix \ref{S:BackgroundMaterial} for details. But this definition of fractional derivative is not the only possible extension, and an alternative definition was formulated much later by Dzhrbashyan and Caputo (who worked independently). This new idea is known in the literature as the Dzhrbashyan-Caputo (or, simply, Caputo for shortness) fractional derivative; for its definition see~\eqref{eq:Caputo_Derivative} in Appendix \ref{S:BackgroundMaterial}.

Each of the RL and Caputo derivatives of real order $\alpha>0$ is a {\em left-inverse} operator for the RL fractional integral. They are each represented as Volterra-like convolution integro-differential operators with integral kernel $k(t) = t^{m-\alpha-1}/\Gamma(m-\alpha)$, where $m=\left\lceil \alpha \right\rceil$ is the smallest integer greater than or equal to $\alpha$. If $\alpha$ is not an integer, this kernel is weakly singular at $t=0$ and it is locally absolutely integrable on the positive real axis. For further discussion we refer the reader to \cite{Diet10,GorenfloMainardi1997,KilbasSrivastavaTrujillo2006,Mainardi2010,Pod99,SKM93}.

The weakly singular nature of the kernel has some important consequences. For example, as shown in~\cite{Sty19}, solutions of time-dependent fractional differential equations (FDEs) with RL or Caputo derivatives typically exhibit weak singularities at the initial time~$t=0$. This phenomenon presents challenging difficulties for the theoretical and numerical analysis of FDEs. One can foresee that numerical difficulties are to be expected because standard numerical methods for solving differential equations are usually based on \emph{polynomial} approximations of the unknown solution --- but polynomials do not provide accurate approximations in the neighbourhood of singularities.
	
In an attempt to avoid the difficulties caused by singularities, some authors have proposed modifications of the RL and Caputo derivatives that are based on the replacement of their weakly singular kernel by some \emph{non-singular} function  that is continuous on the \emph{closed} interval $[0, T]$ with $T>0$. For instance, the exponential function was employed as a replacement for the standard kernel in the Dzhrbashyan-Caputo derivative \eqref{eq:Caputo_Derivative} of order $0<\alpha<1$ to obtain the following so-called Caputo--Fabrizio (CF) derivative:
\begin{equation}\label{eq:CF_Reg_Derivative}
	\DCF_{0}^{\alpha} f(t) = \frac{M(\alpha)}{1-\alpha} \int_{0}^t \exp\Bigl(-\frac{\alpha}{1-\alpha} (t-\tau)\Bigr) f'(\tau) \du \tau ,
\end{equation}
where $M(\alpha)$ is a normalization factor such that $M(0)=M(1)=1$. Similarly, the Mittag-Leffler function
\[
	E_{\alpha}(z) = \sum_{k=0}^{\infty} \frac{z^k}{\Gamma(\alpha k + 1)}
\]
is used instead of the exponential function to define the so-called Atangana-Baleanu (AB) derivative
\begin{equation}\label{eq:AB_Reg_Derivative}
	\DABC_{0}^{\alpha} f(t) = \frac{B(\alpha)}{1-\alpha} \int_{0}^t E_{\alpha}\Bigl(-\frac{\alpha}{1-\alpha} (t-\tau)^{\alpha}\Bigr) f'(\tau) \du \tau ,
\end{equation}
where $B(\alpha)$ has the same role and properties as $M(\alpha)$. We mention for completeness that RL-type versions of the CF and AB operators have also appeared in the literature.

{\sc Note:} For simplicity, in our integral operators we always take the initial time to be $t=0$, since a choice of a different initial time makes no essential difference to the arguments that we shall present.

\vspace*{-3pt}

\begin{remark}
The CF derivative \eqref{eq:CF_Reg_Derivative} has an analytic kernel, while the ABC derivative \eqref{eq:AB_Reg_Derivative} has a kernel that is continuous but not differentiable at~$t=0$. Continuity of the kernel suffices in our discussion, but later in the paper we use its derivative on some open interval $(0,T)$. Anyway, the essence of the integral operators discussed in this paper is that their kernel is bounded in $[0,T]$. Both kernels of CF and ABC derivatives are sufficiently well behaved to satisfy our arguments.
\end{remark}

At first sight, fractional derivatives defined using non-singular kernels may appear very attractive since they  avoid several difficulties that are caused by the singular nature of the RL and Dzhrbashyan-Caputo kernel. Thus, it is unsurprising that these simpler operators have become quite popular since their appearance about five years ago. {\emph{But these operators with non-singular kernels have serious shortcomings that strongly discourage their use}. Some previous papers \cite{Giusti2020,HanygaPrep20,Hanyga20,Kochubei2011,Sty18} have already mentioned some of these; here we attempt to give a comprehensive and coherent account of the drawbacks.

{\sc  Structure of the paper:} In Section \ref{sec:Hanyga} we prove first  that Caputo-type derivatives that are defined using non-singular kernels must fail to satisfy the fundamental theorem of fractional calculus. In other words, they do not allow the existence of a corresponding convolution integral for which the derivative is the left-inverse. Although one can find definitions of ``CF and AB integrals" $\ICF_{0}^{\alpha}$ and $\IAB_{0}^{\alpha}$ in the literature, nevertheless (as we shall see in Section~\ref{S:NotLeftInverse}) the corresponding derivatives do not act on them as a left-inverse, since $\DCF_{0}^{\alpha} \bigl[ \ICF_{0}^{\alpha} f(t) \bigr]  \ne f(t)$ and $\DABC_{0}^{\alpha} \bigl[ \IAB_{0}^{\alpha} f(t) \bigr] \ne f(t)$ unless the unnatural and restrictive condition $f(0)=0$ is imposed on the space of function where these operators act. That is, if these integrals are used to solve differential equations involving CF or ABC derivatives, they can provide a correct solution only when the vector field vanishes at the origin, which is clearly an unreasonable assumption for most problems. In Section~\ref{S:NotLeftInverse} we also show that this issue is shared by any derivative with non-singular kernel since they satisfy a \emph{zero-zero property}: the derivative at 0 is always 0. For time-fractional initial-boundary value problems of parabolic type, one finds that a related restriction is automatically imposed on the initial data --- see Section~\ref{S:IBV_Parabolic}. Moreover, in Section~\ref{S:CF_AB_Not_New} we show that, if one makes the unnatural assumption  that the vector field vanishes at the origin, in order to ensure that CF and ABC derivatives are left-inverse of CF and AB integrals, one then obtains differential equations that are equivalent to standard differential equations of integer or fractional order; in other words, the introduction of these derivatives does not add anything new to the standard (RL and Caputo) theory of fractional calculus.  In Section \ref{S:TrulyDerivatives} we show that derivative-type operators defined by a non-singular kernel fail to satisfy various proposed extensions of the classical notion of derivative, so whether one should use the term ``derivative'' to describe these operators is doubtful.  Some final remarks are given in Section~\ref{S:Concluding}.

{\sc  Notation.} We write $C[0,T]$ for the space of continuous functions on the interval~$[0,T]$ of the real line, and  $L^1[0,T]$ denotes the space of Lebesgue-integrable functions on this interval. The  space of absolutely continuous functions  on $[0,T]$ is denoted by $AC[0,T]$; recall  that $f\in AC[0,T]$ if and only if $f(t) = f(0) + \int_0^t f'(s)\du s$ for $0\le  t \le T$ with some $f'\in L^1[0,T]$.

%%%%%%%%%%%%% Section 2 %%%%%%%%%%%%%%%%%%%%%
\vspace*{-3pt}

\section{The fundamental theorem of fractional calculus}\label{sec:Hanyga}

For the Caputo derivative \eqref{eq:Caputo_Derivative}, by  \cite[Theorem 3.7]{Diet10} one has
\[
	\DC_0^{\alpha} \bigl[ J_0^\alpha f(t) \bigr] = f(t)
\]
for all $f\in C[0,T]$ and $0 <t \le T$. This is a fractional equivalent of (part of) the fundamental theorem of classical calculus. In this section we investigate what conditions a fundamental theorem of fractional calculus imposes on the kernels of  the differential and integral operators. This analysis is based on ideas previously presented in~\cite{Giusti2020,HanygaPrep20,Hanyga20,Kochubei2011}.

First, we state the following well-known technical result, which will be used more than once in this paper.

\begin{lemma}\label{lem:L1cgce}
Let $g\in L^1[0,T]$. Let $\varepsilon >0$ be  given. Then there exists $\delta>0$ such  that $\left| \int_E g(x)\,dx\right| <  \varepsilon$ for every measurable set $E\subset [0,T]$ with measure less than $\delta$.
\end{lemma}

\begin{proof}
See for example  \cite[p.\ 300, Theorem 6]{KolFom75}.
\end{proof}

Suppose that, imitating \eqref{eq:Caputo_Derivative} for $0<\alpha<1$, we define for a function $f \in AC[0,T]$ a Caputo-type derivative $\Dphi$ by
\begin{equation}\label{eq:GeneralDer}
\Dphi f(t) \coloneqq \int_0^t \phi(t-\tau) f'(\tau) \du\tau , \quad 0<t \le T,
\end{equation}
where the kernel function $\phi$  is as yet unspecified, except that we  require $\phi\in L^1[0,T]$ to ensure that $\Dphi f(t)$ is defined almost everywhere (it is well known that  the convolution of two  functions in $L^1[0,T]$ also lies in $L^1[0,T]$; cf., e.g., \cite[proof of Theorem 2.1]{Diet10}).

Such operators based on non-singular kernels usually have a normalization factor ---  see \eqref{eq:CF_Reg_Derivative} and \eqref{eq:AB_Reg_Derivative} for example --- that multiplies the integral, depends on~$\alpha$, and ensures that $\Dphi$ approaches the classical first-order derivative when $\alpha \to 1$. For brevity we do not write this factor explicitly in~\eqref{eq:GeneralDer}; instead it is absorbed into the  kernel $\phi$.

In order to have a fundamental theorem of fractional  calculus for our derivative $\Dphi$, we need to define a  corresponding integral operator $\Jpsi$, defined by
\[
\Jpsi g(t) \coloneqq \int_0^t \psi(t-\tau) g(\tau) \du\tau \ \text{ for } \ 0<t \le T,
\]
where $\psi\in L^1[0,T]$ is yet to be chosen in such a way that  $\Dphi [\Jpsi f(t) ] = f(t)$ for all $f\in AC[0,T]$ and $0 <t \le T$. Writing out this identity in detail, we have
\[
\begin{aligned}
f(t) &= \int_0^t \phi(t-\tau) (\Jpsi f)'(\tau) \du\tau
	= \int_0^t \phi(\tau) (\Jpsi f)'(t-\tau) \du\tau \\
	&= \frac{\du}{\du t} \left\{ \int_0^t \phi(\tau) (\Jpsi f)(t-\tau) \du\tau \right\}, \
\end{aligned}
\]
here the second equation follows from a simple change of variable, while the third is a consequence of Leibniz's Rule for differentiating integrals, combined with $\lim_{t\to 0}\Jpsi f(t)=0$ (which follows from Lemma~\ref{lem:L1cgce} since $\psi\in L^1[0,T]$ and $f$ bounded implies that the integrand of  $\Jpsi f$ lies in $L^1[0,T]$).  Now make another change of variable, then recall the definition of $\Jpsi$ to get
\[
\begin{aligned}
	f(t) &= \frac{\du}{\du t} \left\{ \int_0^t \phi(t-\tau) (\Jpsi f)(\tau) \du\tau \right\} \\
	&= \frac{\du}{\du t} \left\{ \int_{\tau=0}^t \phi(t-\tau)
		\left[ \int_{s=0}^\tau \psi(\tau-s) f(s) \du s\right]  \du\tau \right\}. \
\end{aligned}
\]
Next, apply Fubini's theorem to swap the order of integration, then apply Leibniz's Rule again:
\[
\begin{aligned}
f(t) &=  \frac{\du}{\du t} \left\{ \int_{s=0}^t  f(s)
		\left[  \int_{\tau=s}^t  \phi(t-\tau)\psi(\tau-s) \du \tau\right]  \du  s \right\} \\
	&= f(t)\lim_{s\to t}\left[  \int_{\tau=s}^t  \phi(t-\tau)\psi(\tau-s)  \du \tau\right] + \\
	&\hspace{3.0cm} +  \int_{s=0}^t f(s) \frac{\du}{\du t}\left[  \int_{\tau=s}^t  \phi(t-\tau)\psi(\tau-s)  \du \tau\right]  \du s. \
\end{aligned}
\]

We want this equation to hold true for all $f\in AC[0,T]$ and $0 <t \le T$. This is possible only if
\[
\lim_{s\to t}  \int_{\tau=s}^t  \phi(t-\tau)\psi(\tau-s)  \du \tau =1
\ \text{ and }\
\frac{\du}{\du t}  \int_{\tau=s}^t  \phi(t-\tau)\psi(\tau-s)  \du \tau =0.
\]
The change of variables $r=\tau-s$ shows that each integral here equals $ \int_{r=0}^{t-s}  \phi(t-s-r)\psi(r)  \du r $. Thus, the value of the integral depends on the  length $t-s$ of the interval of integration but not separately on $t$ and $s$. Consequently one can rewrite $\lim_{s\to t}$ in the first condition as $\lim_{t\to s}$.
But the second condition says that $\int_{\tau=s}^t  \phi(t-\tau)\psi(\tau-s)  \du \tau$ is a constant as $t$ varies; then the first condition forces
\begin{equation}\label{Sonine}
\int_{\tau=s}^t  \phi(t-\tau)\psi(\tau-s)  \du \tau = 1 \ \text{ for } \ 0 \le s < t \le T.
\end{equation}

Equations of the form \eqref{Sonine} are known as Sonine equations. They impose a certain requirement on  the  functions  $\phi$ and $\psi$ (which, up  to now, were merely required to lie in $L^1[0,T]$) and their interaction.  Suppose that  one of these functions is bounded on~$[0,T]$; say, $|\phi(t)|\le M$ for $0\le t \le T$. Then
\[
\left|\int_{\tau=s}^t  \phi(t-\tau)\psi(\tau-s)  \du \tau\right| \le M \int_{\tau=s}^t |\psi(\tau-s)|  \du \tau,
\]
and by Lemma~\ref{lem:L1cgce} the right-hand side will go to zero if $s\to t$. But  this implies that the Sonine  equation \eqref{Sonine} cannot  be satisfied when $s$ is close to $t$. Thus  we cannot  have $\phi$  bounded on~$[0,T]$ (and likewise for $\psi$).

\smallskip

The above argument can be summarised as follows:

\begin{theorem}
Given a Caputo-type fractional derivative of the form \eqref{eq:GeneralDer} whose kernel $\phi : [0,T] \to \mathbb R$ is bounded, one cannot define a corresponding integral operator such that the fundamental theorem of fractional calculus is valid.
\end{theorem}

In particular, this theorem applies to kernels that are continuous functions on~$[0,T]$.

A similar result for fractional derivatives of RL-type is derived in~\cite{HanygaPrep20,Hanyga20}.

%%%%%%%%%%%%%  Section 3 %%%%%%%%%%%%%%%%%%%%%%%%

\section{Derivatives with non-singular kernel impose restrictive \break and unnatural assumptions}\label{S:NotLeftInverse}

In this section we show that differential equations involving derivatives with a non-singular kernel  impose severe (and unnatural) constraints on the initial conditions.

We do this by first considering differential equations involving the CF and ABC derivatives, and then moving on to the general case of non-singular kernels, which we analyse using Laplace transforms. The discussion in this section is for initial-value problems  posed on $[0,T]$; a related restriction for initial-boundary  value problems will be presented in Section~\ref{S:IBV_Parabolic}.

\subsection{CF and ABC derivatives are not the left-inverse of the corresponding integrals}\label{SS:CFABnot} %% 3.1 %%

A so-called CF integral $\ICF^{\alpha}_0$ has been proposed in the literature, defined by
\begin{equation}\label{eq:CF_Integral}
	\ICF^{\alpha}_0 f(t) = \frac{1-\alpha}{M(\alpha)} f(t) + \frac{ \alpha}{M(\alpha)} \int_0^t f(\tau) \du \tau , \quad t \ge 0.
\end{equation}
It has the property that  $\ICF^{\alpha}_0 \bigl[ \DCF_0^{\alpha} f(t) \bigr] = f(t) - f(0)$; that is,  the differential operator is the \emph{right}-inverse of the integral operator on the space of functions $\{f \in AC[0,T]: f(0) = 0\}$. This property is similar to the identity $\int_0^t  f'(s)\du s = f(t)-f(0)$ enjoyed by classical first-order derivatives and the standard integral operator. But first-order derivatives also have the \emph{left}-inverse property $\frac{\du}{\du t}\int_0^t  f(s)\du s = f(t)$, whereas for the CF integral and derivative we have the following result.

\begin{proposition}	\label{thm:ICF_NotRightInverse_DCF}
Let $f \in AC[0,T]$.	The CF derivative and the CF integral satisfy the relation
	\begin{equation}\label{eq:ICF_NotRightInverse_DCF}
		\DCF_0^{\alpha} \bigl[ \ICF^{\alpha}_0 f(t) \bigr] = f(t) - \exp\Bigl(-\frac{\alpha}{1-\alpha} t\Bigr) f(0) .
	\end{equation}
\end{proposition}

\begin{proof}
See Appendix \ref{S:ProofDerivativeNotRightInverseIntegral}.
\end{proof}

This unfavourable result says that the CF derivative $\DCF_0^{\alpha}$ is the left-inverse of $\ICF^{\alpha}_0$ only on the restricted space $\{f \in AC[0,T]: f(0) = 0\}$ and not on the full space $AC[0,T]$, as one would expect (and as is the case for the Caputo derivative \cite[Theorem 3.7]{Diet10}).

The constraint $f(0)=0$ on functions for which $\DCF_0^{\alpha}$ is the left-inverse of $\ICF^{\alpha}_0$ has serious consequences if $\ICF^{\alpha}_0$ is employed to solve an initial-value problem such as
\begin{equation}\label{eq:CF_FDE_Eq}
	\left\{\begin{array}{l}
	\DCF_0^{\alpha} y(t) = g(t,y(t)), \\	
	y(0) = y_0 .\
	\end{array}\right.
\end{equation}

For applying $\ICF^{\alpha}_0$ to both sides of this differential  equation, one obtains
\begin{equation}\label{eq:CF_FDE_Eq_Sol}
	y(t) = y_0 +  \frac{1-\alpha}{M(\alpha)} g(t,y(t)) + \frac{ \alpha}{M(\alpha)} \int_0^t g(\tau,y(\tau)) \du \tau.
\end{equation}
But, replacing $f(t)$  in (\ref{eq:ICF_NotRightInverse_DCF}) by~$g(t,y(t))$, we see immediately that
\[
	\DCF_0^{\alpha} y(t) = g(t,y(t)) - \exp\Bigl(-\frac{\alpha}{1-\alpha} t\Bigr) g(0,y_0).
\]
Hence $\DCF_0^{\alpha} y(t) \not= g(t,y(t))$ if $g(0,y_0)\not=0$.  That is, although one  might believe erroneously that $y(t)$ in~\eqref{eq:CF_FDE_Eq_Sol} is the solution of~(\ref{eq:CF_FDE_Eq}), this is not true unless $g(0,y_0)=0$.

The situation is  similar for  the so-called AB integral
\begin{equation}\label{eq:AB_Integral}
	 \IAB_0^{\alpha} f(t) = \frac{1-\alpha}{B(\alpha)} f(t) + \frac{ \alpha}{B(\alpha)\Gamma(\alpha)} \int_0^t (t-\tau)^{\alpha-1} f(\tau) \du \tau , \quad t \ge 0 .
\end{equation}
Here again $\IAB^{\alpha}_0 \bigl[ \DABC_0^{\alpha} f(t) \bigr] = f(t) - f(0)$, but $\DABC_0^{\alpha}$ is not the left-inverse of $\IAB^{\alpha}_0$ since the following analog of Theorem \ref{thm:ICF_NotRightInverse_DCF} holds (see  Appendix \ref{S:ProofDerivativeNotRightInverseIntegral} for the proof):

\begin{proposition}	\label{thm:IAB_NotRightInverse_DABC}
Let $f \in AC[0, T]$. The ABC derivative and the AB integral satisfy the relation
	\begin{equation}
		\label{eq:IAB_NotRightInverse_DABC}
		\DABC_0^{\alpha} \bigl[ \IAB^{\alpha}_0 f(t) \bigr] = f(t) - E_{\alpha} \Bigl(-\frac{\alpha}{1-\alpha} t^{\alpha} \Bigr) f(0) .
	\end{equation}
\end{proposition}

Like the CF derivative, the ABC derivative is the left-inverse of the AB integral only on the restricted space $\{f \in AC[0,T]: f(0) = 0\}$.

The use of $\IAB^{\alpha}_0$ to solve a differential equation with the ABC derivative of the same type of (\ref{eq:CF_FDE_Eq}) will thus produce a function $y(t)=y_0 + \IAB_0^{\alpha}g(t,y(t))$ that is not a solution of the equation since
\[
\DABC_0^{\alpha}  y(t) = g(t,y(t)) - E_{\alpha}\Bigl(-\frac{\alpha}{1-\alpha} t^{\alpha} \Bigr) g(0,y_0) .
\]

In general the CF and AB integrals cannot be used to solve differential equations with the corresponding fractional derivatives, unless one imposes the additional and restrictive condition $g(0,y_0)=0$ to have the identities $\DCF_0^{\alpha} \bigl[ \ICF^{\alpha}_0 g(t,y(t)) \bigr] = g(t,y(t))$ and $\DABC_0^{\alpha} \bigl[ \IAB^{\alpha}_0 g(t,y(t)) \bigr] = g(t,y(t))$.

To appreciate how unnatural the condition $g(0,y_0)=0$ is, consider the simple linear problem where $g(t,y(t)) = \lambda y(t)$ in~(\ref{eq:CF_FDE_Eq}) with a CF or ABC derivative. Imposing the condition $g(0,y_0)=0$, so that the CF or AB integral solves the problem correctly, requires either $\lambda=0$ or $y_0=0$; but then the problem has only the trivial constant solution $y(t)\equiv y_0$ for all $t \ge 0$. Introducing new operators only  to describe constant solutions is not worthwhile!

Propositions~\ref{thm:ICF_NotRightInverse_DCF} and~\ref{thm:IAB_NotRightInverse_DABC}
will be generalized in Theorem~\ref{thm:Jpsi} of Section~\ref{SS:DifferentIntegral}.

\subsection{Non-singular kernel derivatives are always zero at zero}\label{SS:ZeroZero} %% 3.2 %%

The restriction on the initial condition of differential equations with CF and ABC derivatives is consequence of the fact the these derivatives are zero at the origin. For instance, taking the power function $f(t)=t^{\gamma}$ for constant $\gamma>0$,  one can compute
%\[
%	\DCF_0^{\alpha} f(t) =  \gamma t^{\gamma} E_{1,\gamma+1}\Bigl(-\frac{\alpha}{1-\alpha} t \Bigr) , \quad
%	\DABC_0^{\alpha} f(t) = \gamma t^{\gamma} E_{\alpha,\gamma+1}\Bigl(-\frac{\alpha}{1-\alpha} t^{\alpha} \Bigr) ,
%\]
%\[
%	\DCF_0^{\alpha} f(t) =  \textstyle\frac{n! M(\alpha) t^{n}}{1-\alpha} E_{1,n+1}\bigl(-\textstyle\frac{\alpha}{1-\alpha} t \bigr) , \quad
%	\DABC_0^{\alpha} f(t) = \textstyle\frac{n! B(\alpha) t^{n}}{1-\alpha} E_{\alpha,n+1}\bigl(-\textstyle\frac{\alpha}{1-\alpha} t^{\alpha} \bigr) ,
%\]
\[
\DABC_0^{\alpha} f(t) = \frac{B(\alpha) }{1-\alpha} \Gamma(\gamma+1)  t^{\gamma} E_{\alpha,\gamma+1}\Bigl(-\frac{\alpha}{1-\alpha} t^{\alpha} \Bigr) ,
\]
and consequently $\DABC_0^{\alpha} f(t) \bigl|_{t=0} = 0$ (similarly for $\DCF_0^{\alpha} f(t)$). We call this the \emph{zero-zero property} (namely, the derivative at 0 is always 0). It holds true not only for  CF and ABC derivatives, and not only for the function $f(t)=t^{\gamma}$, but much more generally, as we now show.
%and consequently $\DCF_0^{\alpha} f(t) \bigl|_{t=0} = 0$ and $\DABC_0^{\alpha} f(t) \bigl|_{t=0} = 0$. We call this the \emph{zero-zero property} (namely, the derivative at 0 is always 0). It holds true not only for  CF and ABC derivatives, and not only for the function $f(t)=t^{\gamma}$, but much more generally, as we now show.

\begin{theorem}[Zero-zero property]\label{thm:NonSingDer_ZeroZeroProp_LT}
Let $\phi$ be bounded on $[0,T]$, $\Dphi$ the operator defined by (\ref{eq:GeneralDer}) and $f\in AC[0,T]$. Then
\[
	\lim_{t\to 0^{+}} \Dphi f(t) = 0.
\]
\end{theorem}

\begin{proof}
Since $\phi$  is bounded on $[0,T]$, for any $t \in (0,T]$ one has
\[
	\bigl| \Dphi f(t) \bigr| = \left| \int_{0}^{t} \phi(t-\tau) f'(\tau) \du \tau \right|
		\le \left(\sup_{t \in [0,T]}|\phi(t)|\right) \int_{0}^{t} | f'(\tau) | \du \tau .
\]
But $f \in AC[0,T]$ means that $f' \in L^{1}[0,T]$, so Lemma~\ref{lem:L1cgce} implies the desired result.
\end{proof}

\begin{remark}
The argument used to prove Theorem~\ref{thm:NonSingDer_ZeroZeroProp_LT} fails for the Caputo derivative~\eqref{eq:Caputo_Derivative}, because then the kernel blows up as $t\to 0^+$ and consequently does not have a maximum value; the function $f(t)=t^\alpha$ is a counterexample.
\end{remark}

Consider now a general differential equation, with a non-singular (i.e. bounded) kernel derivative $D_{\phi}$, of the form
\begin{equation}\label{eq:NonSingDer_General_Eq}
	\left\{\begin{array}{l}
	D_{\phi} y(t) = g(t,y(t)) \\	
	y(0) = y_0 \
	\end{array}\right. ,
\end{equation}
for which Theorem \ref{thm:NonSingDer_ZeroZeroProp_LT} gives $0 =D_{\phi} y(t)\bigl|_{t=0^{+}} = g(0,y_0)$. Hence (\ref{eq:NonSingDer_General_Eq}) can have a solution only if  $g(0,y_0)=0$.

Thus in (\ref{eq:NonSingDer_General_Eq})  one is forced to choose the initial data $y_0$ such that $g(0,y_0)=0$. This is of course restrictive --- and may  even be impossible in some cases.

\subsection{No inversion of non-singular kernel derivatives without restrictions}\label{SS:DifferentIntegral} %% 3.3 %%

Overlooking the \emph{zero-zero property}  of derivatives with bounded kernel (Theorem~\ref{thm:NonSingDer_ZeroZeroProp_LT}) can lead to the construction of integral operators, such as  $\ICF^{\alpha}_0$ and $\IAB^{\alpha}_0$, that are sometimes misinterpreted as inverse operators for the corresponding derivatives. For when these integral operators are applied to solve the differential equation (\ref{eq:NonSingDer_General_Eq}) they do not yield correct solutions unless one imposes restrictions on the data, as we shall show in Theorem~\ref{thm:Jpsi}, which generalizes  Propositions~\ref{thm:ICF_NotRightInverse_DCF} and~\ref{thm:IAB_NotRightInverse_DABC}.

We work in the following general framework. The bounded kernel $\phi(t)$ that defines the non-singular derivative $\Dphi$ is usually defined for all $t\ge 0$, but the problems that we consider are typically posed on a bounded interval~$[0,T]$. Thus we regard $\phi$ as defined only on $[0,T]$, and for the purpose of taking its Laplace Transform (LT) we extend $\phi(t)$ to $(0, \infty)$ by setting $\phi(t) = 0$ for all~$t>T$. This extension (or any other extension of $\phi$ on $(T,\infty)$) does not affect the differential equations that we investigate.

\smallskip

Then we make the following  assumptions on the kernel function $\phi$ :
\begin{itemize}
	\item[{\bf H1}:] $\phi(t)$ is continuous on $[0,T]$;
	\item[{\bf H2}:] $\phi(t)$ is differentiable on $(0, T)$ and $\phi'(t)$ has at worst an integrable singularity at $t=0$.
\end{itemize}

The assumptions {\bf H1} and  {\bf H2} are not restrictive; for example, they are satisfied by the CF and AB kernels.

The LT is defined in the usual way: for all suitable functions $g$ and $s>0$, the LT of $g$ is
\[
\hat g(s) := \int_{t=0}^\infty \eu^{-st}g(t)\,  \du t,
\]
the assumptions {\bf H1} --{\bf H2} and our zero extension of $\phi(t)$ from $[0,T]$ to $[0,\infty)$  ensure that the LT~$\hat \phi(s)$  of $\phi$ exists for all~$s>0$.

Set $u(t) = \Dphi f(t) = \int_0^t \phi(t-\tau) f'(\tau) \du\tau$ for $0<t \le T$. Then standard LT properties (see, e.g., \cite[Section D.3]{Diet10}) give
\[
	\hat{u}(s) = \hat{\phi}(s) \bigl[ s \hat{f}(s) - f(0) \bigr],
\]
where  $\hat{f}(s)$ and $\hat{u}(s)$ are the LT of $f(t)$ and  $u(t)$ respectively. Hence
\begin{equation}\label{eq:NonSingDer_General_Invers1}
\hat{f}(s) = \frac{1}{s} f(0) + \hat{\psi}(s) \hat{u}(s), \ \text{ where }\  \hat{\psi}(s) := \frac{1}{s \hat{\phi}(s)} .
\end{equation}
But the final value theorem \cite[Theorem D.13]{Diet10} for the LT and \textbf{H1} yields
\begin{equation}\label{eq:FiniteLimitOrigin}
\displaystyle\lim_{s\to \infty} s \hat{\phi}(s) = \lim_{t \to 0^{+}} \phi(t) = \phi(0).
\end{equation}
\noindent
Consequently $\lim_{s\to \infty} \hat{\psi}(s) = 1/\phi(0) \not=0$. It then follows from  \cite[Theorem 23.2]{Doetsch1974} that $\hat{\psi}(s)$ cannot be the  LT of any function $\psi(t)$. Thus, one cannot invert the LT in (\ref{eq:NonSingDer_General_Invers1}) to obtain a solution of the form $f(t) = f(0) + {\int_0^t  \psi (t-\tau) u(\tau) \du \tau}$.

\smallskip

One might  try to circumvent this obstacle by the following device:  set
\[
\hat{\psi}^{\star}(s) = \hat{\psi}(s) - \frac{1}{\phi(0)}
\]
\noindent
and reformulate (\ref{eq:NonSingDer_General_Invers1})  as
\begin{equation}\label{eq:NonSingDer_General_Invers2}
	\hat{f}(s) = \frac{1}{s} f(0) + \frac{1}{\phi(0)}\hat{u}(s) +  \hat{\psi}^{\star}(s) \hat{u}(s) .
\end{equation}
Since $\lim_{\Re(s)\to \infty} \hat{\psi}^{\star}(s) = 0$, one cannot exclude {\em a priori}
the existence of a function $\psi^{\star}(t)$ whose LT is $\hat{\psi}^{\star}(s)$. If such a function exists, one can transform (\ref{eq:NonSingDer_General_Invers2}) back to the time domain, obtaining  $f(t) = f(0) + \tilde{J}_{\psi} u(t)$, where
\begin{equation}\label{eq:GeneralIntegr}
	\tilde{J}_{\psi} u(t) = \frac{1}{\phi(0)} u(t) + \int_{0}^{t} \psi^{\star}(t - \tau) u(\tau) \du \tau.
\end{equation}
Thus we now have an operator $\tilde{J}_{\psi}$, analgous to $\ICF^{\alpha}_0$ and $\IAB^{\alpha}_0$ in Section~\ref{SS:CFABnot}, such that
\[
	\tilde{J}_{\psi} \bigl[ \Dphi f(t) \bigr] = f(t) - f(0) .
\]
It turns out however that $D_{\phi}$ is not necessarily the left inverse of  $\tilde{J}_{\psi}$, as we now show.
\begin{theorem} \label{thm:Jpsi} %% Th 3.2 %%
Let $\tilde{J}_{\psi}$ be the operator defined in (\ref{eq:GeneralIntegr}) and $f \in AC[0,T]$. Then
\[
	\Dphi \bigl[ \tilde{J}_{\psi} f(t) \bigr] =  f(t) - \frac{\phi(t)}{\phi(0)} f(0) - \phi(t) \cdot  \lim_{t\to 0^{+}} J_{\psi^{\star}}f(t) ,
\]
where $J_{\psi^{\star}}f(t) = \int_0^t \psi^{\star}(t-\tau) f(\tau) \du \tau$.
\end{theorem}

\begin{proof}
It is straightforward to evaluate
\[
\begin{aligned}
	\Dphi \bigl[ \tilde{J}_{\psi} f(t) \bigr]
	&= \int_0^{t} \phi(t-\tau) \frac{\du }{\du \tau} \left[ \frac{1}{\phi(0)} f(\tau) + \int_{0}^{\tau} \psi^{\star}(\tau - u) f(u) \du u \right] \du \tau \\
	&= \frac{1}{\phi(0)} \Dphi f(t) + \int_0^{t} \phi(t-\tau) g(\tau) \du \tau,
\end{aligned}
\]
where $g(t) := \frac{\du }{\du t} \int_{0}^{t} \psi^{\star}(t-\tau) f(\tau) \du \tau$. Then the LT of $g$  is
\[
	\hat g(s)
	= s \hat{\psi}^{\star}(s) \hat{f}(s) - J_{\psi^{\star}}f(t)\bigl|_{t=0}
	= \frac{1}{\hat{\phi}(s)} \hat{f}(s) - \frac{s}{\phi(0)} \hat{f}(s) - \lim_{t\to 0^{+}} J_{\psi^{\star}}f(t).
\]
\noindent
Hence
\[
	\begin{aligned}
	{\mathcal L } \Bigl( \int_0^{t} &\phi(t-\tau) g(\tau) \du \tau \, ; \, s \Bigr)
	=
	\hat{\phi}(s) \left[ \frac{1}{\hat{\phi}(s)} \hat{f}(s) - \frac{s}{\phi(0)} \hat{f}(s) - \lim_{t\to 0^{+}} J_{\psi^{\star}}f(t) \right] \\
	&= \hat{f}(s) - \frac{ \hat{\phi}(s) \bigl(s \hat{f}(s) - f(0)\bigr)}{\phi(0)} - \frac{\hat{\phi}(s)}{\phi(0)} f(0) - \hat{\phi}(s) \cdot \lim_{t\to 0^{+}} J_{\psi^{\star}}f(t) .\\
	\end{aligned}
\]
Inverting the LT, we get
\[
\int_0^{t} \phi(t-\tau) g(\tau) \du \tau
= f(t) - \frac{1}{\phi(0)} \Dphi f(t) - \frac{\phi(t)}{\phi(0)} f(0) - \phi(t) \cdot \lim_{t\to 0^{+}} J_{\psi^{\star}}f(t)
\]
from which the result follows.
\end{proof}

For well behaved functions it is possible that $\lim_{t\to 0^{+}} J_{\psi^{\star}}f(t)=0$; but then, whenever $f(0)\not=0$, it is clear that $\Dphi \bigl[ \tilde{J}_{\psi} f(t) \bigr] \not=  f(t)$. That is, the operator $\tilde{J}_{\psi} u(t)$ does not give a solution of the differential equation (\ref{eq:NonSingDer_General_Eq}) unless, once again, the restriction  $g(0,y_0)=0$  is imposed on the vector field $g(t,y(t))$.

%%%%%%%%%%%%% Section 4 %%%%%%%%%%%%%%%%%%%%%%%%%

\section{Parabolic time-fractional initial-boundary value problems}\label{S:IBV_Parabolic}

In this section we follow~\cite{Sty18}.
Let $\Omega$ be a bounded domain in $\mathbb{R}^n$ for some $n\ge 1$. Let $T>0$ be fixed. We consider initial-boundary value problems posed on $\Omega \times [0,T]$.

Let $\alpha\in (0,1)$.
For any suitable function $g(x,t)$ defined on $\Omega\times [0,T]$,  the Caputo fractional temporal derivative $D^\alpha_t$ of order $\alpha$  is  (see~\eqref{eq:Caputo_Derivative}):
\[
D_t^\alpha g(x,t) := \frac1{\Gamma(1-\alpha)}\int_{\tau=0}^t (t-\tau)^{-\alpha}\,\frac{\partial g(x,\tau)}{\partial \tau}\du \tau,  \quad\text{for } x\in\Omega, \ 0< t \le T.
\]
Consider the time-fractional initial-boundary value problem
\begin{subequations} \label{prob}
\begin{align}
D_t^\alpha u &- \Delta u = f(x,t)\label{proba}
\intertext{for $(x,t)\in Q:=\Omega\times(0,T]$, with}
u(x,t)&= g(x,t)\quad\text{for }(x,t) \in \partial\Omega \times (0,T], \label{probb} \\
u(x,0)&= u_0(x) \quad\text{for } x\in\bar\Omega, \label{probc}
\end{align}
\end{subequations}
where  the given functions $g$ and $u_0$ are continuous on the closures of their domains.

Suppose now that $D_t^\alpha g$ is replaced by
\begin{equation}\label{contK}
\tilde D_t^\alpha g(x,t) :=  \int_{\tau=0}^t K(t,\tau) \,\frac{\partial g(x,\tau)}{\partial \tau}\du \tau  \quad\text{for } x\in\Omega, \, 0< t \le T,
\end{equation}
where \emph{the kernel  $K(t,\tau)$ is nonsingular,} i.e., $K$ is continuous on $[0,T]\times[0,T]$.
(Here, similarly to~\eqref{eq:GeneralDer}, we do not write down any explicit normalisation factor for~ $D_t^\alpha$; this factor is absorbed into the  kernel $K$.) Note that the kernel $K(t,\tau)$ includes kernels of the form $\phi(t-\tau)$  as a special case.

\smallskip

For this nonsingular kernel, one has the following remarkable result:

\begin{theorem}\cite[Theorem 1]{Sty18} \label{th:1Sty18} %% Th. 4.1 %%
Let $u(x,t)$ be a solution of the initial-boundary value problem \eqref{prob}, where a continuous-kernel fractional derivative $\tilde D_t^\alpha u$ is used in~\eqref{proba}. Suppose that for each $x\in \Omega$, the function  $u(x,\cdot)$ lies in $AC[0,T]$. Then the initial data~$u_0(x) = u(x, 0)$ must satisfy the equation $\Delta u_0(x) = f(x,0)$ on $\Omega$.
\end{theorem}

\begin{proof}
Theorem~\ref{thm:NonSingDer_ZeroZeroProp_LT} implies that  $\lim_{t\to 0^+}\tilde D_t^{\alpha}u(x,t)=0$ for each $x\in\Omega$. Hence, taking the limit of   equation~\eqref{proba} as $t\to 0^+$, we get $\Delta u_0(\cdot) = f(\cdot,0)$ on~$\Omega$.
\end{proof}

\begin{remark}
The hypothesis of Theorem~\ref{th:1Sty18} that  for each $x\in \Omega$, the function  $u(x,\cdot)$ lies in $AC[0,T]$
 is not restrictive. This condition is satisfied by almost every example in the literature on time-fractional initial-boundary value problems.
\end{remark}

The next example shows the powerful consequences of Theorem~\ref{th:1Sty18}.

\begin{example} \label{exa:2Sty18}%% Ex. 4.1 %%
Consider the fractional heat equation
\[
\tilde D_t^\alpha v - \partial^2 v/\partial x^2 = 0\ \text{ for } (x,t)\in (0,1)\times(0,T],
\]
\noindent
where $\tilde D_t^\alpha$ is a continuous-kernel fractional derivative, the boundary data are $v(0,t) = v(1,t) =0$ and the initial data $v(x,0) = v_0(x)$, where $v_0(x) \in C^2[0,1]$ is unspecified except that it satisfies the initial-boundary compatibility condition $v_0(0) = v_0(1) = 0$, so that any solution~$u$ of~\eqref{prob} is continuous on $\bar\Omega \times [0,T]$.

Assume that for each $x$, the solution $v(x,\cdot)$ of this problem lies in $L^1[0,T]$. Then Theorem~\ref{th:1Sty18} and the above compatibility condition show  that $v_0$ must satisfy the conditions
\[
- v_0''(x) = 0 \ \text{ on } (0,1), \quad v_0(0) = v_0(1) = 0.
\]
But these conditions imply that $v_0 \equiv 0$. As all the data of this example are now zero, we get $v\equiv 0$.

Thus, using a continuous-kernel fractional derivative forces the problem to have as its solution $v\equiv 0$; the apparent freedom of choice that one has for~$v_0$ is only an illusion.
\end{example} %%%%%%%%

In~\cite{Sty18}, the differential operator $-\Delta u$ of~\eqref{proba} is replaced by a much more general spatial operator, and  it is shown that under reasonable conditions, \emph{the initial data \eqref{probc} is determined uniquely by the other data of the problem}. Example~\ref{exa:2Sty18} is a particular case of this phenomenon. Such a restrictive condition is  extremely unnatural and it is clearly caused by the use of a continuous-kernel fractional derivative.

%%%%%%%%%%% Section 5 %%%%%%%%%%%%%%%%%%%%

\section{Are CF and ABC derivatives really new and necessary?}\label{S:CF_AB_Not_New}

As we saw in Section \ref{S:NotLeftInverse}, the differential equation \eqref{eq:NonSingDer_General_Eq} with the CF and ABC derivatives (or any other non-singular kernel derivative) requires  $g(0,y_0)=0$ in order to have a solution, where $g(t,y(t))$ is the vector field of the differential equation.

Suppose that one limits the use of these derivatives to those problems that satisfy the condition $g(0,y_0)=0$. We now show that in this special case, the CF and ABC derivatives serve no purpose since the problem can then be described by simpler operators.

Consider the CF initial-value problem \eqref{eq:CF_FDE_Eq}. Assume  that $g(0,y_0)=0$. Then the solution of the problem is (\ref{eq:CF_FDE_Eq_Sol}). Differentiating this equation gives
%% \vskip -10pt
\[
	\frac{\du}{\du t} y(t) = \frac{1-\alpha}{M(\alpha)} \frac{\du}{\du t} g(t,y(t)) + \frac{ \alpha}{M(\alpha)} g(t,y(t))
\]
--- so $y(t)$ is the solution of an integer-order differential equation! Thus there is no need to use a fractional derivative to find $y$; this function can be handled in the framework of classical calculus.

\smallskip

This observation that the  CF derivative is not truly fractional but can be reformulated using integer-order derivatives is discussed in \cite{Tar19}.

Similarly, for the ABC derivative, under the assumption that $g(0,y_0)=0$ one gets
\[
	\DC^{\alpha}_0 y(t) = \frac{1-\alpha}{B(\alpha)} \DC^{\alpha}_0 g(t,y(t)) + \frac{ \alpha}{B(\alpha)} g(t,y(t)),
\]
so $y$ is the solution of a Caputo differential equation and introducing the ABC derivative is unnecessary.

A further connection between CF and ABC derivatives and some standard operators of integer and fractional-order operators was shown in \cite{Giusti18}.

\begin{remark}
The CF and ABC derivatives are sometimes described as special cases of the fractional Prabhakar derivative, but this is not true.  Introduced in \cite{DOvidioPolito2018} to provide a Caputo-like regularization of the operator previously introduced in \cite{KilbasSaigoSaxena2004}, the Prabhakar derivative is defined as
\[
	\DPrC^{\gamma}_{\alpha, \beta, \lambda; 0}f(t) = 	\int_{0}^{t}(t-u)^{m-\beta-1}E^{-\gamma}_{\alpha,m-\beta}\left( \lambda (t-u)^{\alpha} \right)  f^{(m)}(u) \, \du u ,
\]
with $m=\left\lceil \beta \right\rceil$. It is not obtained by replacing the standard power law kernel of the Dzhrbashyan-Caputo derivative with a particular realization of the three-parameter ML function \cite{Prabhakar1971}
\[
E_{\alpha,\beta}^{\gamma}(z) = \frac{1}{\Gamma(\gamma)}\sum_{k=0}^{\infty} \frac{\Gamma(\gamma+k) z^{k}}{k! \Gamma(\alpha k + \beta)}
		, \quad
		\alpha, \beta, \gamma \in \Cset
		, \quad \real(\alpha) > 0 , \,\,\,  z \in \Cset \, ,
\]
\noindent
but rather it is defined as the left-inverse of the Prabhakar integral
\[			
	\JPr^{\gamma}_{\alpha,\beta, \lambda; 0}f(t)
	=\int_{0}^{t}(t-u)^{\beta-1}E^{\gamma}_{\alpha,\beta}\left(\lambda (t-u)^{\alpha} \right)f(u) \du u,  \quad \alpha,\beta>0 \, .
\]
Hence, unlike CF and ABC operators, the Prabhakar derivative $\DPrC^{\gamma}_{\alpha, \beta, \lambda; 0}$ naturally satisfies the fundamental theorem of fractional calculus. Moreover, the kernel of $\DPrC^{\gamma}_{\alpha, \beta, \lambda; 0}f(t)$ is always singular at $t=0$ (except for the limit case $\beta \in \Nset$ discussed in details in \cite{Giusti2020}) and no zero-zero property holds with the Prabhakar derivative. The standard Dzhrbashyan-Caputo derivative of order $\beta$ is obtained when $\gamma=0$ or $\lambda=0$. We refer to \cite{GiustiColombaroGarraGarrappaPolitoPopolizioMainardi2020} for a complete treatment of the Prabhakar fractional calculus.
\end{remark}

%%%%%%%%%%% Section 6 %%%%%%%%%%%%%%%%%%%%%

\section{Are non-singular kernel derivatives really derivatives?}\label{S:TrulyDerivatives}

One should also consider whether operators obtained by inserting a non-singular kernel in the RL and Caputo derivatives can really be described as derivatives. While several  papers make systematic attempts to determine whether or not a new operator is fractional \cite{HilfLuch19,OrtTen15,OrtTen18,Tar13,Tar18}, to the best of our knowledge very few attempts have been made to discern whether or not an operator is a derivative. There are  contributions by Ortigueira and Machado \cite{OrtigueiraMachado2019,OrtigueiraMachado2020} based on systems theory, but we wish to explore this question using only mathematical considerations.

In our Appendix~\ref{SS:IndirectProcess} we describe the \emph{indirect process} for the derivation of  the RL and Caputo derivatives. In this process one first generalises integer-order repeated integrals to any real positive order, then one defines derivatives as operators that are the inverse of the integral (by analogy with integer-order calculus, where the derivative can be viewed as the inverse operator of the integral). The RL (\ref{eq:RL_Derivative}) and Caputo (\ref{eq:Caputo_Derivative}) derivatives obtained in this way are formulated in terms of convolution integrals. Under assumptions that are reasonable and unrestrictive, they are equivalent to operators obtained by a more straightforward generalization of the usual definition of the integer order derivative (see the description of the \emph{direct process} in Appendix~\ref{SS:DirectProcess}).

Defining derivatives by means of integrals may appear unnatural at first sight, but the property of acting as an inverse
of the repeated integral,  and the fact that the same operators can be obtained by generalizing the integer-order derivative, together provide a compelling justification for recognising  the RL and Caputo derivative operators as bona fide derivatives.

The construction of fractional derivatives with non-singular kernels imitates, but only in a partial way, the indirect process described above.
It might appear attractive to modify the RL and Caputo derivatives by replacing their singular kernels by a non-singular function, but then it is difficult to justify the statement that these new operators are really derivatives. Indeed:
\begin{itemize}
	\item as we have shown in Sections~\ref{sec:Hanyga} and~\ref{S:NotLeftInverse}, there is no integral of which non-singular kernel derivatives are the inverse operators; thus the \emph{indirect process} leading to the construction of the RL and Caputo derivatives is partly imitated but is not fully replicated;
	\item there is no evidence that these new ``derivatives" can be generated through a direct generalisation of the integer-order derivative, as described in Appendix \ref{SS:DirectProcess} for the RL and Caputo fractional derivatives.
\end{itemize}

Consequently we think that it is truly questionable to describe as derivatives the operators discussed in this paper. Formulas such as (\ref{eq:CF_Reg_Derivative}) and (\ref{eq:AB_Reg_Derivative}), and more generally (\ref{eq:GeneralDer}), are more akin to integral operators than to  derivatives and calling them derivatives is misleading; to avoid confusion, the more general term \emph{operator} rather than the specific term \emph{derivative} should be used.

\subsection{A further observation} %% 6.1 %%
Integration by parts of (\ref{eq:CF_Reg_Derivative}) and (\ref{eq:AB_Reg_Derivative}) yields
\[
	\begin{aligned}
	\DCF^{\alpha}_{0} f(t) = \frac{M(\alpha)}{1-\alpha} \Biggl[ f(t) - \exp&\Bigl(-\frac{\alpha}{1-\alpha} t\Bigr) f(0) \\
	&- \frac{\alpha}{1-\alpha}\int_{0}^t \exp\Bigl(-\frac{\alpha}{1-\alpha} (t-\tau)\Bigr) f(\tau) \du \tau\Biggr]
	\end{aligned}
\]
\noindent
and
\[
	\begin{aligned}
	\DABC^{\alpha}_{0} f(t) = \frac{B(\alpha)}{1-\alpha} \Biggl[ &f(t) - E_{\alpha}\Bigl(-\frac{\alpha}{1-\alpha} t^{\alpha}\Bigr) f(0) \\
	&- \frac{\alpha}{1-\alpha}\int_{0}^t (t-\tau)^{\alpha-1} E_{\alpha,\alpha}\Bigl(-\frac{\alpha}{1-\alpha} (t-\tau)^{\alpha}\Bigr) f(\tau) \du \tau\Biggr] \
	\end{aligned}
\]
(note that a similar calculation is impossible for derivatives with a singular kernel, such as the  RL and Caputo derivatives). These identities surely cast doubt on any claim that $\DCF^{\alpha}_{0}$ and $\DABC^{\alpha}_{0}$  represent derivatives, since they merely comprise evaluations of the function~$f$ and a weighted integral of~$f$, i.e., no differentiation of~$f$ is involved.

%%%%%%%%%%% Section 7 %%%%%%%%%%%%%%%%%%%%%%%%%%%%%

\section{Concluding remarks}\label{S:Concluding}

In this paper we have discussed the properties and drawbacks of the operators, commonly called non-singular kernel derivatives, that are obtained by replacing the singular kernel of the Caputo derivative with a non-singular function. While the so-called CF and ABC derivatives are the best-known operators of this type, our analysis covers any derivative with non-singular kernel.

We have shown that non-singular kernel derivatives do not in general have an  inverse that can  be written as  a convolution integral --- unlike the RL and Caputo derivatives, which enjoy this property. One can construct an integral with the operator as its left-inverse only if the function is zero at the origin. This follows from a ``zero-zero" property: non-singular kernel derivatives are always zero when evaluated at the initial time~$t=0$. Consequently, it is possible to solve differential equations with non-singular kernel derivatives only when a very restrictive and unnatural assumption is made on the initial condition.

We then go on to show that if one accepts this restrictive condition, then the CF and ABC derivatives can be replaced by finite combinations of operators that are already known from classical calculus and the Caputo derivative calculus.

We also cast doubt on the belief that a  non-singular kernel derivative can  be  regarded as a true form of derivative.

Our overwhelming conclusion from all this evidence is that derivatives with non-singular kernel should never be used.

\appendix %%%%%%%%%%%%%%%%%%%%%%%%%%%%%%%%%%%%%%%%%5

\section{Background material on fractional calculus}\label{S:BackgroundMaterial} %% App. A %%

Standard fractional derivatives such as the RL and Dzhrbashyan-Caputo derivatives can be introduced by following a \emph{direct process} that starts from the integer-order derivative and leads to a fractional generalisation of the difference quotient. Alternatively, by an \emph{indirect process} one first obtains the RL integral as a generalisation of the usual integer-order repeated integral, then inverses of this integral, which are formulated in terms of a convolution integrals, are defined as fractional derivatives.

For completeness of exposition we briefly describe here the two processes and show that they lead to equivalent operators.

\subsection{Indirect process: generalization of integer-order integrals and inversion}\label{SS:IndirectProcess} %% A.1 %%

To introduce fractional derivatives, begin by considering the standard (integer-order) $n$-fold repeated integral
\begin{equation}\label{eq:nInteger_Integral}
	\begin{aligned}
	J_{0}^{n} f(t) \coloneqq
	 \frac{1}{(n-1)!}\int_{0}^{t} (t-\tau)^{n-1} f(\tau) \du \tau
	, \quad t > 0. \\
	\end{aligned}
\end{equation}
Then the Euler-Gamma function $\Gamma(x)= \int_0^{\infty}t^{x-1} \eu^{-t} \du t$, $\Re(x)>0$, which satisfies $\Gamma(n) = (n-1)!$ for $n\in\Nset$, allows us to extend~\eqref{eq:nInteger_Integral} from integers~$n$  to any real positive number~$\alpha$ by setting
\vskip -10pt
\begin{equation}\label{eq:RL_Integral}
 \IRL_{0}^{\alpha} f(t) \coloneqq \frac{1}{\Gamma(\alpha)} \int_{{0}}^{t}
(t-\tau)^{\alpha-1} f(\tau) \du \tau
	, \quad  t > 0.
\end{equation}
This is the fractional Riemann-Liouville integral.

A left-inverse  of the operator $\IRL_{0}^{\alpha}$ is an operator that when applied to $\IRL_{0}^{\alpha} f(t)$ gives back the original function $f(t)$. It is possible to find more than one such  operator. In fact, writing $m= \left\lceil \alpha \right \rceil$ for the smallest integer greater than or equal to $\alpha$ and $D^{m}$  for the usual integer-order differentiation, both the fractional RL derivative
\begin{equation}\label{eq:RL_Derivative}
	\DR_{0}^{\alpha} f(t) \coloneqq
	\frac{1}{\Gamma(m-\alpha)} \frac{\du^{m}}{\du t^{m}} \int_{{0}}^{t} (t-\tau)^{m-\alpha-1} f(\tau) \du \tau,  \quad
	t > 0
\end{equation}
\noindent
and the Caputo derivative
\begin{equation}\label{eq:Caputo_Derivative}
	\DC_{0}^{\alpha} f(t) \coloneqq
	\frac{1}{\Gamma(m-\alpha)} \int_{0}^{t} (t-\tau)^{m-\alpha-1} f^{(m)}(\tau) \du \tau , \quad
	t > 0, 	
\end{equation}
are left-inverses of $\IRL_{0}^{\alpha}$ since $\DR_{0}^{\alpha} \bigl[ \IRL_{0}^{\alpha} f \bigr]  = \DC_{t_0}^{\alpha} \bigl[ \IRL_{0}^{\alpha} f \bigr] = f$ under reasonable hypotheses on~$f$ (see, e.g., \cite[Theorems 2.14 and 3.7]{Diet10}).
It is well known  \cite[Definition 3.2]{Diet10} that these two operators are related by
%% \vskip -10pt
\begin{equation}\label{eq:RL_Caputo_Relationship}
	\DC_{0}^{\alpha} f(t) = \DR_{0}^{\alpha} \Bigl[ f(t) - T_{m-1}[f,0](t) \Bigr],
\end{equation}
\noindent
where $T_{m-1}[f,0](t)$ is the Taylor polynomial of $f(t)$ expanded around $0$, viz.,
\[
	T_{m-1}[f,0](t) = \sum_{k=0}^{m-1} \frac{t^{k}}{k!} f^{(k)}(0).
\]

\subsection{Direct process: generalization of integer-order derivatives}\label{SS:DirectProcess} %% A.2 %%

To describe a more \emph{direct process} we first consider the usual definition of the first-order derivative
\begin{equation}\label{eq:IntegerOrderDerivativeFirst}
	f'(t) = \lim_{h\to 0^+} \frac{f(t) - f(t-h)}{h} ,
\end{equation}
which is easily extended to any $n \in \Nset$ by simple recursion to obtain
\begin{equation}\label{eq:IntegerOrderDerivative}
	f^{(n)}(t) = \lim_{h\to0^+} \frac{1}{h^n} \sum_{j=0}^{n} \omega_{j}^{(n)} f(t-jh)
	, \quad
	\omega_{j}^{(n)} = (-1)^j \binom{n}{j}.
\end{equation}
(For ease of presentation we take into consideration only limits from the right).
Once again appealing to the Euler-Gamma function, the binomial coefficients can be reformulated as
\begin{equation}\label{eq:BinomialCoefficientsInteger}
	\binom{n}{j} =  \displaystyle\frac{n!}{j!(n-j)!} = \left\{\begin{array}{ll}
		\displaystyle\frac{\Gamma(n+1)}{j!\Gamma(n+1-j)} \quad & j=0,1,\dots,n , \\
		0 & j > n .\\
	\end{array}\right.
\end{equation}
\noindent
Then the  simple observation that $\omega_{j}^{(n)} = 0$ for any $j > n$ allows us to rewrite (\ref{eq:IntegerOrderDerivative}) as an infinite series
\[
	f^{(n)}(t) = \lim_{h\to0^+}  \frac{1}{h^n} \sum_{j=0}^{\infty} \omega_{j}^{(n)} f(t-jh)
	, \quad
\]
\noindent
and, since (\ref{eq:BinomialCoefficientsInteger}) permits coefficients $\omega_{j}^{(n)}$ with  $n$ replaced by a non-integer parameter $\alpha>0$, a generalisation to fractional order of  (\ref{eq:IntegerOrderDerivative})  is easily obtained:
\begin{equation}\label{eq:FOD_GL}
	\DG^{\alpha} f(t) = \! \lim_{h\to0^+} \! \frac{1}{h^{\alpha}} \sum_{j=0}^{\infty} \omega_{j}^{(\alpha)} f(t-jh)
	, \quad
	\omega_{j}^{(\alpha)} = (-1)^j\frac{\Gamma(\alpha+1)}{j! \Gamma(\alpha-j+1)} .
\end{equation}

The operator $\DG^{\alpha}$ is generally known as the Gr\"unwald-Letnikov (GL) fractional derivative because proposed independently by Gr\"unwald~\cite{Grunwald1867} and Letnikov~\cite{Letnikov1868} in 1867 and 1868 respectively.  Although it is perhaps the most straightforward generalization of the integer-order derivative to any fractional order, it has some drawbacks:
\begin{itemize}
	\item $\DG^{\alpha}f(t)$ requires the knowledge of the whole history of the function~$f$ in $(-\infty,t]$:  while this may not be a difficulty from a purely mathematical point of view when $f(t)$ is known analytically for all~$t$, when $\DG^{\alpha}$ is applied in differential equations the solution $f(t)$ (usually the state of a system) is known only starting from a given initial time. For this reason the use of $\DG^{\alpha}$ is mainly confined to signals theory, where signals are often decomposed into $\sin$ and $\cos$ functions whose values are available for all~$t$;
	\item the series in (\ref{eq:FOD_GL}) converges only for a restricted class of functions (for instance, bounded functions when $0<\alpha<1$) and this limitation is too  restrictive for general applications.
\end{itemize}

To overcome these drawbacks, a common strategy  is to fix a starting point, say for convenience $0$, and impose suitably chosen values for $f(t)$ on $(-\infty,0)$. Usually the following functions are considered:
\[
	\fR(t) = \left\{ \begin{array}{ll}
		0 & t \in (-\infty, 0) \\
		f(t) \, \, & t \ge 0 \
	\end{array}\right.
	, \quad
	\fC(t) = \left\{ \begin{array}{ll}
		T_{m-1}[y,0](t) & t \in (-\infty, 0) \\
		f(t) \, \, & t \ge 0 . \
	\end{array}\right.
\]
The replacement of $f$ by $\fR$ or $\fC$, together with the property of the coefficients $\omega_{j}^{(\alpha)}$ (see, e.g., \cite{GarrappaKaslikPopolizio2019}) that $\sum_{j=0}^{\infty} \omega_{j}^{(\alpha)} j^{k} = 0$ for $k=0,1,\dots,m-1$, leads to the  two distinct fractional derivatives
\[
	\begin{aligned}
	&\DG^{\alpha} \fR(t) = \lim_{h\to0^+} \frac{1}{h^{\alpha}} \sum_{j=0}^{N} \omega_{j}^{(\alpha)} f(t-jh)
	, \quad \\
	&\DG^{\alpha} \fC(t) = \lim_{h\to0^+} \frac{1}{h^{\alpha}} \sum_{j=0}^{N} \omega_{j}^{(\alpha)} \Bigl[ f(t-jh) - T_{m-1}[y,0](t-jh) \Bigr] , \\
	\end{aligned}
\]
\noindent
where $N=\left\lfloor t/h\right\rfloor$.

\smallskip

Interestingly, there is a link between the indirect and direct processes. A result in fractional calculus \cite[Theorem 2.25]{Diet10} states that if $f \in C^{m}[0,T]$, then $\DG^{\alpha} \fR(t) = \DR_{0}^{\alpha} f(t)$; consequently, in view of (\ref{eq:RL_Caputo_Relationship}), one also has $	\DG^{\alpha} \fC(t) = \DC_{0}^{\alpha} f(t)$.

%%%%%%%%%%%%%%%%%%%%%%%%%%%%  App. B %%%%%%%%%%%%%

\section{CF and ABC  derivatives are not left-inverse \break
         of CF and AB integrals}\label{S:ProofDerivativeNotRightInverseIntegral}

For completeness, in this section we give the elementary derivations showing that the CF derivative $\DCF_0^{\alpha}$ is not the left-inverse of the so-called CF integral $\ICF^{\alpha}_0$. For notational convenience, throughout this section we use the abbreviation
\[
	W_{\alpha} = \frac{\alpha}{1-\alpha} .
\]

\bigskip

%\begin{proof}[Proof of Proposition \ref{thm:ICF_NotRightInverse_DCF}]
P\,r\,o\,o\,f\, o\,f\, P\,r\,o\,p\,o\,s\,i\,t\,i\,o\,n\,  \ref{thm:ICF_NotRightInverse_DCF}.\,
Let $f \in AC[0,T]$.  Observe that	
\[
	\DCF_0^{\alpha} \bigl[ \ICF^{\alpha}_0 f(t) \bigr]
	= \underbrace{\frac{1-\alpha}{M(\alpha)} \DCF_0^{\alpha}  f(t)}_{(A)} + \underbrace{\frac{ \alpha}{M(\alpha)} \DCF_0^{\alpha}  \int_0^t f(s) \du s}_{(B)}  .
\]	
Integration by parts allows us to evaluate the first integral (A):
\[
	\begin{aligned}
	(A) &= \int_0^t \exp\Bigl(-W_{\alpha} (t-\tau)\Bigr) f'(\tau) \du \tau \\
	& = f(t) - \exp\Bigl(-W_{\alpha} t\Bigr) f(0) - W_{\alpha} \int_0^t \exp\Bigl(-W_{\alpha} (t-\tau)\Bigr) f(\tau) \du \tau .
	\end{aligned}
\]
For the integral (B), one has immediately
\[
	\begin{aligned}
	(B)
	&= W_{\alpha} \int_0^t \exp\Bigl(-W_{\alpha} (t-\tau)\Bigr) \frac{\du}{\du \tau} \int_0^{\tau} f(s) \du s \, \du \tau \\
	&= W_{\alpha} \int_0^t \exp\Bigl(-W_{\alpha} (t-\tau)\Bigr)  f(\tau) \du \tau .
	\end{aligned}
\]
Now we are done, since $\DCF_0^{\alpha} \bigl[ \ICF^{\alpha}_0 f(t) \bigr] = (A) + (B)$.
%\end{proof}
\hfill $\square$ %%%%%%%%%%%%%%%%%%%%%%%%%

\bigskip

In a similar way, one can show that the ABC derivative $\DABC_0^{\alpha}$ is not the left-inverse of the so-called AB integral $\IAB^{\alpha}_0$.

\bigskip

%\begin{proof}[Proof of Proposition \ref{thm:IAB_NotRightInverse_DABC}]
P\,r\,o\,o\,f\, o\,f\, P\,r\,o\,p\,o\,s\,i\,t\,i\,o\,n\,  \ref{thm:IAB_NotRightInverse_DABC}.\,
Let $f \in AC[0,T]$.  Observe that
$$
	\DABC_0^{\alpha} \bigl[ \IAB_0^{\alpha} f(t) \bigr]
	= \underbrace{\frac{1-\alpha}{B(\alpha)} \DABC_0^{\alpha} f(t)}_{(C)} + \underbrace{\frac{ \alpha}{B(\alpha)} \DABC_0^{\alpha} \int_0^t \frac{(t-\tau)^{\alpha-1}}{\Gamma(\alpha)} f(\tau) \du \tau}_{(D)} .
$$
Integration by parts yields
\[
	\begin{aligned}
	(C)	&= \int_0^t E_{\alpha}\Bigl(-W_{\alpha} (t-\tau)^{\alpha}\Bigr) f'(\tau) \du \tau \\
	&=  f(t) - E_{\alpha}\Bigl(-W_{\alpha } t^{\alpha} \Bigr) f(0)
	 - W_{\alpha}\int_0^t \tau^{\alpha-1} E_{\alpha,\alpha}\Bigl(-W_{\alpha } \tau^{\alpha} \Bigr) f(t-\tau) \du \tau.
	\end{aligned}
\]
\noindent
To evaluate (D) we first observe that
\[
(D) =  W_{\alpha} \int_0^t E_{\alpha}\Bigl(-W_{\alpha} (t-\tau)^{\alpha}\Bigr) \frac{\du }{\du \tau} \IRL^{\alpha}_0 f(\tau)
\]
and by means of the LT we can evaluate \cite[Eq. (1.10)]{Mainardi2010}
\[
	{\mathcal L} \left( \frac{\du }{\du \tau} \IRL^{\alpha}_0 f(\tau) \du \tau \, ; \, s \right) = s \frac{1}{s^{\alpha}} \hat{f}(s) - \IRL^{\alpha}_0 f(0^+) = s^{1-\alpha} \hat{f}(s) ,
\]
where we used standard rules for the LT of the first-order derivative together with $\IRL^{\alpha}_0 f(0^+)=0$ since $f \in AC[0,T]$.  Therefore, since the LT of $t^{\beta-1} E_{\alpha,\beta}(-\lambda t^\alpha )$ is $s^{\alpha-\beta}/(s^{\alpha}+\lambda)$ \cite[Eq. (4.10.1)]{GorenfloKilbasMainardiRogosin2014}, we  get
\[
{\mathcal L} \left( \int_0^t E_{\alpha}\Bigl(-W_{\alpha} (t-\tau)^{\alpha}\Bigr) \frac{\du }{\du \tau} \IRL^{\alpha}_0 f(\tau)  \, ; \, s \right) = \frac{1}{s^{\alpha} + W_{\alpha}}  \hat{f}(s)  .
\]
The inversion of the LT gives
\[
	(D) = W_{\alpha} \int_0^t \tau^{\alpha-1} E_{\alpha,\alpha}\Bigl(-W_{\alpha} \tau^{\alpha}\Bigr) f(t-\tau) \du \tau
\]
\noindent
from which the result follows since $\DABC_0^{\alpha} \bigl[ \IAB_0^{\alpha} f(t) \bigr] = (C) + (D)$.
%\end{proof}
\hfill $\square$ %%%%%%%%%%%%%%%%%%%%%%%%%%

\section*{Acknowledgments}	%%%%%%%%%%%%%%%%

The cooperation which has lead to this paper has been initiated and promoted within the COST Action CA15225, a network supported by COST (European Cooperation in Science and Technology). The work of Kai Diethelm is also supported by the German Federal Ministry of Education and Research (BMBF) under Grant No.\ 01IS17096A. The work of Roberto Garrappa is also supported under a GNCS-INdAM 2020 Project. Andrea Giusti is supported by the Natural Sciences and Engineering Research Council of Canada (Grant No.~2016-03803 to V. Faraoni) and by Bishop's University. The research of Martin Stynes is supported in part by the National Natural Science Foundation of China under grant NSAF U1930402.

\vskip 1 cm


\begin{thebibliography}{99} %%%%%%%%%%%%%%%%%
\normalsize

\bibitem{Diet10}
K.~Diethelm, \emph{The Analysis of Fractional Differential Equations}, Volume
  2004 of {Lecture Notes in Mathematics}. Springer-Verlag, Berlin (2010).

\bibitem{Doetsch1974}
G.~Doetsch, \emph{Introduction to the Theory and Application of the {L}aplace
  Transformation}. Springer-Verlag, New York-Heidelberg (1974).

\bibitem{DOvidioPolito2018} 
M.~D'Ovidio, F.~Polito, {Fractional diffusion--telegraph equations and their
  associated stochastic solutions}. \emph{Theory Probab. Appl.} \textbf{62},
  No~4 (2018), 552--574 [appeared as an \emph{arXiv Preprint}, arXiv:1307.1696,
  in 2013].

\bibitem{GarrappaKaslikPopolizio2019}
R.~Garrappa, E.~Kaslik, M.~Popolizio, Evaluation of fractional integrals and
  derivatives of elementary functions: Overview and tutorial.
  \emph{Mathematics} \textbf{7}, No~5 (2019), \# 407; doi:{10.3390/math7050407}.

\bibitem{Giusti18}
A.~Giusti, A comment on some new definitions of fractional derivative.
  \emph{Nonlinear Dyn.} \textbf{93} (2018), 1757--1763;
  doi:{10.1007/s11071-018-4289-8}.

\bibitem{Giusti2020}
A.~Giusti, General fractional calculus and {P}rabhakar's theory. \emph{Commun.
  Nonlinear Sci. Numer. Simul.} \textbf{83} (2020), \# 105114, 7; \hfill \break
  doi:{10.1016/j.cnsns.2019.105114}.

\bibitem{GiustiColombaroGarraGarrappaPolitoPopolizioMainardi2020}
A.~Giusti, I.~Colombaro, R.~Garra, R.~Garrappa, F.~Polito, M.~Popolizio,
  F.~Mainardi, A practical guide to {P}rabhakar fractional calculus.
  \emph{Fract. Calc. Appl. Anal.} \textbf{23}, No~1 (2020), 9--54,
  doi:{10.1515/fca-2020-0002}; \hfill \break
  \href{https://www.degruyter.com/view/journals/fca/23/1/fca.23.issue-1.xml}
  {https://www.degruyter.com/view/journals/fca/23/1/} \hfill \break
  \hspace*{0.3cm} \hfill {fca.23.issue-1.xml}.

\bibitem{GorenfloKilbasMainardiRogosin2014}
R.~Gorenflo, A.~A. Kilbas, F.~Mainardi, S.~V. Rogosin, \emph{Mittag-{L}effler
  Functions, Related Topics and Applications}. Springer Monographs in
  Mathematics, Springer, Heidelberg (2014).

\bibitem{GorenfloMainardi1997}
R.~Gorenflo, F.~Mainardi, Fractional calculus: integral and differential
  equations of fractional order. In: \emph{Fractals and Fractional Calculus in
  Continuum Mechanics ({U}dine, 1996)}, Volume 378 of \emph{CISM Courses and
  Lect.}, 223--276, Springer, Vienna (1997).

\bibitem{Grunwald1867}
A.~Gr\"unwald, {\"Uber} ``begrenzte'' {Derivationen} und deren {Anwendung}.
  \emph{Z. Angew. Math. Phys.} \textbf{12} (1867), 441--480.

\bibitem{HanygaPrep20}
A.~Hanyga, A simple proof that a generalized fractional derivative cannot have
  a regular kernel. \emph{ResearchGate Preprint}: \hfill \break
  \href{https://www.researchgate.net/publication/339353727}{https://www.researchgate.net/publication/339353727}. %\url{https://www.researchgate.net/publication/339353727}.

\bibitem{Hanyga20}
A.~Hanyga, A comment on a controversial issue: {A} generalized fractional
  derivative cannot have a regular kernel. \emph{Fract. Calc. Appl. Anal.}
  \textbf{23}, No~1 (2020), 211--223, doi:{10.1515/fca-2020-0008};
  \href{https://www.degruyter.com/view/journals/fca/23/1/fca.23.issue-1.xml}
  {https://www.degruyter.com/view/journals/fca/23/1/} \hfill \break
  \hspace*{0.3cm} \hfill {fca.23.issue-1.xml}.

\bibitem{HilfLuch19}
R.~Hilfer, Y.~Luchko, Desiderata for fractional derivatives and integrals.
  \emph{Mathematics} \textbf{7}, No~2 (2019), \# 149; doi:{10.3390/math7020149}.

\bibitem{KilbasSaigoSaxena2004}
A.A. Kilbas, M.~Saigo, R.K. Saxena, Generalized {M}ittag-{L}effler function
  and generalized fractional calculus operators. \emph{Integr. Transf.
  Spec. Funct.} \textbf{15}, No~1 (2004), 31--49; % \hfill \break % 
  doi:{10.1080/10652460310001600717}.

\bibitem{KilbasSrivastavaTrujillo2006}
A.A. Kilbas, H.M. Srivastava, J.J. Trujillo, \emph{Theory and Applications
  of Fractional Differential Equations}, Volume 204 of {North-Holland
  Mathematics Studies}. Elsevier Science B.V., Amsterdam (2006).

\bibitem{Kochubei2011}
A.N. Kochubei, General fractional calculus, evolution equations, and renewal
  processes. \emph{Integral Equations Operator Theory} \textbf{71}, No~4
  (2011), 583--600; doi:{10.1007/s00020-011-1918-8}.

\bibitem{KolFom75}
A.N. Kolmogorov, S.V. Fom\={\i}n, \emph{Introductory Real Analysis}. Dover
  Publications, Inc., New York (1975).

\bibitem{Letnikov1868}
A.~Letnikov, Theory of differentiation with an arbitrary index (in {Russian}).
  \emph{Mat. Sbornik} \textbf{3}, No~1 (1868), 1--68.

\bibitem{Mainardi2010}
F.~Mainardi, \emph{Fractional Calculus and Waves in Linear Viscoelasticity}.
  Imperial College Press, London (2010).

\bibitem{OrtTen15}
M.D. Ortigueira, J.A. Tenreiro Machado, What is a fractional derivative?
  \emph{J. Comput. Phys.} \textbf{293} (2015), 4--13;  \hfill \break 
  doi:{10.1016/j.jcp.2014.07.019}.

\bibitem{OrtTen18}
M.D. Ortigueira, J.A. Tenreiro Machado, A critical analysis of the
  {C}aputo-{F}abrizio operator. \emph{Commun. Nonlinear Sci. Numer. Simul.}
  \textbf{59} (2018), 608--611; doi:{10.1016/j.cnsns.2017.12.001}.

\bibitem{OrtigueiraMachado2019}
M.D. Ortigueira, J.A. Tenreiro Machado, Fractional derivatives: The
  perspective of system theory. \emph{Mathematics} \textbf{7}, No~2 (2019),
  \# 150.

\bibitem{OrtigueiraMachado2020}
M.D. Ortigueira, J.A. Tenreiro Machado, On the properties of some operators
  under the perspective of fractional system theory. \emph{Commun. Nonlinear
  Sci. Numer. Simul.} \textbf{82} (2020), \# 105022, 8; \hfill \break 
  doi:{10.1016/j.cnsns.2019.105022}.

\bibitem{Pod99}
I.~Podlubny, \emph{Fractional Differential Equations}, Volume 198 of
  {Mathematics in Science and Engineering}. Academic Press, Inc., San
  Diego, CA (1999).

\bibitem{Prabhakar1971}
T.R. Prabhakar, A singular integral equation with a generalized
  {M}ittag--{L}effler function in the kernel. \emph{Yokohama Math. J.}
  \textbf{19}, No~1 (1971), 7--15.

\bibitem{SKM93}
S.G. Samko, A.A. Kilbas, O.I. Marichev, \emph{Fractional Integrals and
  Derivatives}. Gordon and Breach Science Publishers, Yverdon (1993).

\bibitem{Sty18}
M.~Stynes, Fractional-order derivatives defined by continuous kernels are too
  restrictive. \emph{Appl. Math. Lett.} \textbf{85} (2018), 22--26; \hfill \break
  doi:{10.1016/j.aml.2018.05.013}.

\bibitem{Sty19}
M.~Stynes, Singularities. In: \emph{Handbook of Fractional Calculus with
  Applications. {V}ol. 3}, 287--305, De Gruyter, Berlin (2019); \hfill \break
  doi:{10.1515/9783110571684}.

\bibitem{Tar13}
V.E. Tarasov, No violation of the {L}eibniz rule. {N}o fractional derivative.
  \emph{Commun. Nonlinear Sci. Numer. Simul.} \textbf{18}, No~11 (2013),
  2945--2948; doi:{10.1016/j.cnsns.2013.04.001}.

\bibitem{Tar18}
V.E. Tarasov, No nonlocality. {N}o fractional derivative. \emph{Commun.
  Nonlinear Sci. Numer. Simul.} \textbf{62} (2018), 157--163;  \hfill \break
   doi:{10.1016/j.cnsns.2018.02.019}.

\bibitem{Tar19}
V.E. Tarasov, Caputo-{F}abrizio operator in terms of integer derivatives:
  memory or distributed lag? \emph{Comput. Appl. Math.} \textbf{38}, No~3
  (2019), Art. 113, 15; doi:{10.1007/s40314-019-0883-8}.

\end{thebibliography}
\end{document}